\documentclass[11pt]{amsart}

\usepackage{amssymb}

\newcommand{\Q}{\mathbb{Q}}
\newcommand{\F}{\mathbb{F}}
\newcommand{\Fp}{{\F_p}}
\newcommand{\Fq}{{\F_q}}
\newcommand{\Z}{\mathbb{Z}}
\newcommand{\D}{\mathcal{D}}

\newcommand{\gp}{\mathfrak{p}}
\newcommand{\gP}{\mathfrak{P}}

\newcommand{\lra}{\longrightarrow}

\newcommand{\from}{\colon}

\newcommand{\chr}{\mathrm{char}}
\newcommand{\Gal}{\mathrm{Gal}}
\newcommand{\Perm}{\mathrm{Perm}}
\newcommand{\Tr}{\mathrm{Tr}}
\newcommand{\End}{\mathrm{End}}

\newtheorem{theorem}{Theorem}
\newtheorem{proposition}{Proposition}[section]
\newtheorem{lemma}[proposition]{Lemma}
\newtheorem{corollary}[proposition]{Corollary}

\theoremstyle{definition}

\newtheorem*{remark}{Remark}

\numberwithin{equation}{section}

\begin{document}

\title[A Valuation Criterion for Normal Basis Generators]
{A Valuation Criterion for Normal Basis  \\
  Generators of Hopf-Galois Extensions in \\
Characteristic $p$}    

\author{Nigel P.~Byott}

\address{%
Mathematics Research Institute \\
University of Exeter \\
Exeter EX4 4QF \\
UK}

\email{N.P.Byott@ex.ac.uk}



\date{February 7, 2011}

\begin{abstract}
Let $S/R$ be a finite extension of discrete valuation rings of
characteristic $p>0$, and suppose that the corresponding extension
$L/K$ of fields of fractions is separable and is $H$-Galois for some
$K$-Hopf algebra $H$. Let $\D_{S/R}$ be the different of $S/R$. We
show that if $S/R$ is totally ramified and its degree $n$ is a power
of $p$, then any element $\rho$ of $L$ with $v_L(\rho) \equiv
-v_L(\D_{S/R})-1 \pmod{n}$ generates $L$ as an $H$-module. This
criterion is best possible. These results generalise to the Hopf-Galois 
situation recent work of G.~G.~Elder for Galois extensions.
\end{abstract}

\maketitle

\section{Introduction}

Let $L/K$ be a finite Galois extension of fields with Galois group
$G=\Gal(L/K)$. The Normal Basis Theorem asserts that there is an
element $\rho$ of $L$ whose Galois conjugates $\{\sigma(\rho) \mid
\sigma \in G\}$ form a basis for the $K$-vector space $L$.
Equivalently, $L$ is a free module of rank 1 over the group algebra
$K[G]$ with generator $\rho$.  Such an element $\rho$ is called a
normal basis generator for $L/K$. The question then arises whether
there is a simple condition on elements $\rho$ of $L$ which guarantees
that $\rho$ is a normal basis generator. Specifically, suppose that
$L$ is equipped with a discrete valuation $v_L$. (Throughout, whenever
we consider a discrete valuation $v_F$ on a field $F$, we assume it is
normalised so that $v_F(F)=\Z \cup \{\infty\}$.) We may then ask
whether there exists an integer $b$ such that any $\rho \in L$ with
$v_L(\rho)=b$ is automatically a normal basis generator for $L/K$.  We
shall refer to any such $b$ as an {\em integer certificate} for normal
basis generators of $L/K$. In the case that $K$ has characteristic
$p>0$, and is complete with perfect residue field, this question was
recently settled by G.~Elder \cite{E}. His result can be stated as
follows:

\begin{theorem}[Elder] \label{elder}
Let $K$ be a field of characteristic $p>0$, complete with respect to
the discrete valuation $v_K$, and with perfect residue field. Let $L$
be a finite Galois extension of $K$ of degree $n$ with Galois group
$G=\Gal(L/K)$, let $w=v_L(\D_{L/K})$, where $\D_{L/K}$ denotes the
different of $L/K$ and $v_L$ is the valuation on $L$, and
let $b \in \Z$.
\begin{enumerate}
\item[(a)] If $L/K$ is totally ramified, $n$ is a power of $p$, and $b
  \equiv -w-1 \pmod{n}$, then every $\rho 
\in L$ with $v_L(\rho)=b$ is a normal basis generator for $L/K$.
\item[(b)] The result of (a) is best possible in the sense that, if 
\begin{enumerate}
\item[(i)] $n$ is not a power of $p$, or
\item[(ii)] $L/K$ is not totally ramified, or
\item[(iii)] $b \not \equiv -w-1 \pmod{n}$,
\end{enumerate}
then there is some $\rho \in L$ with $v_L(\rho)=b$ such that $\rho$ is
{\em not} a normal basis generator for $L/K$
\end{enumerate}
\end{theorem}

The purpose of this paper is to show that Theorem \ref{elder},
suitably interpreted, applies not just in the setting of classical
Galois theory, but also in the setting of Hopf-Galois theory for
separable field extensions, as developed by C.~Greither and
B.~Pareigis \cite{GP}. A finite separable field extension $L/K$ is
said to be $H$-Galois, where $H$ is a Hopf algebra over $K$, if $L$ is
an $H$-module algebra and the map $H \lra \End_K(L)$ defining the
action of $H$ on $L$ extends to an $L$-linear isomorphism $L \otimes_K
H \lra \End_K(L)$. A Hopf-Galois structure on $L/K$ consists of a
$K$-Hopf algebra $H$ and an action of $H$ on $L$ so that $L$ is
$H$-Galois. This generalises the classical notion of Galois
extension: if $L/K$ is a finite Galois extension of fields with Galois
group $G$, we can take $H$ to be the group algebra $K[G]$ with its
standard Hopf algebra structure and its natural action on $L$, and
then $L/K$ is $H$-Galois. A Galois extension may, however, admit many
other Hopf-Galois structures in addition to this classical one, and
many (but not all) separable extensions which are not Galois
nevertheless admit one or more Hopf-Galois structures. Moreover, if
$L$ is $H$-Galois, then $L$ is a free $H$-module of rank 1 (see the
proof of \cite[(2.16)]{C}), and, by analogy with the classical case,
we will shall refer to any free generator of the $H$-module $L$ as a
normal basis generator for $L/K$ with respect to $H$. Our main result
is that Theorem 1 holds in this more general setting:

\begin{theorem} \label{main}
Let $S/R$ be a finite extension of discrete valuation rings of
characteristic $p>0$, and let $L/K$ be the corresponding extension of
fields of fractions. Let $n=[L:K]$, let $v_L$ be the
valuation on $L$ associated to $S$, and let $w=v_L(\D_{S/R})$ where
$\D_{S/R}$ denotes the different of $S/R$. Suppose that $L/K$ is
separable, and is $H$-Galois for some $K$-Hopf algebra $H$. 
Let $b\in \Z$. 
\begin{enumerate}
\item[(a)] If $L/K$ is totally ramified, $n$ is a power of $p$, and $b
\equiv -w-1 \pmod{n}$, then every $\rho \in L$ with $v_L(\rho)=b$ is
a normal basis generator for $L/K$ with respect to $H$.
\item[(b)] The result of (a) is best possible in the sense that, if 
\begin{enumerate}
\item[(i)] $n$ is not a power of $p$, or
\item[(ii)] $L/K$ is not totally ramified, or
\item[(iii)] $b \not \equiv -w-1 \pmod{n}$,
\end{enumerate}
then there is some $\rho \in L$ with $v_L(\rho)=b$ such that $\rho$ is
{\em not} a normal basis generator for $L/K$ with respect to $H$.
\end{enumerate}
\end{theorem}

In Theorem \ref{main}, we do not require $K$ to be complete with
respect to the valuation $v_K$ on $K$ associated to $R$, and we do not
require the residue field of $R$ to be perfect. Thus, even in the case
of Galois extensions (in the classical sense), Theorem \ref{main} is slightly
stronger than Theorem \ref{elder}.

We recall that the different $\D_{S/R}$ is defined as the fractional
$S$-ideal such that
$$ \D_{S/R}^{-1} = \{ x \in S \mid \Tr_{L/K}(xS) \subseteq R \}, $$
where $\Tr_{L/K}$ is the trace from $L$ to $K$. In the case that $S/R$
is totally ramified and $L/K$ is separable, let $p(X) \in R[X]$ be the
minimal polynomial over $R$ of a uniformiser $\Pi$ of $S$. Then
$\D_{S/R}$ is generated by $p'(\Pi)$, where $p'(T)$ denotes the
derivative of $p(T)$ \cite[III, Cor.~2 to Lemma 2]{Se-CL}. (This does
not require $L/K$ to be Galois, or the residue field of $K$ to be
perfect.)  The formulation of Theorem \ref{elder}(a) in \cite{E} is in
terms of $p'(\Pi)$.

If $S$ (and hence $L$) is complete with respect to $v_L$, then
$\D_{S/R}$ is the same as the different $\D_{L/K}$ of the extension
$L/K$ of valued fields occurring in Theorem \ref{elder}. Theorem
\ref{main} also applies, however, if $K$ is a global function field of
dimension 1 over an arbitrary field $k$ of characteristic $p$. In
particular, if $L$ is an $H$-Galois extension of $K$ of $p$-power
degree, and some place $\gp$ of $K$ is totally ramified in $L/K$, then
Theorem \ref{main}(a) gives an integer certificate for normal basis
generators of $L/K$ with respect to $H$, in terms of the valuation
$v_L$ on $L$ corresponding to the unique place $\gP$ of $L$ above
$\gp$ and the $\gP$-part of $\D_{L/K}$. If, on the other hand, there
is more than one place $\gP$ of $L$ above $\gp$, then the integral
closure of $R$ in $L$ is the intersection $S_0$ of the corresponding
valuation rings $S$ of $L$ \cite[III.3.5]{St}. Any one such $S$
strictly contains $S_0$ and is therefore not integral over $R$. In
particular, $S$ is not finite over $R$ and Theorem \ref{main} does not
apply in this case.
 
We briefly recall the background to the above results.  In the
(characteristic 0) situation where $K$ is a finite extension of the field
$\Q_p$ of $p$-adic numbers, the author and Elder \cite{BE} showed the
existence of integer certificates for normal basis generators in
totally ramified elementary abelian extensions $L/K$, under the
assumption that $L/K$ contains no maximally ramified subfield. This
assumption is necessary, since there can be no integer certificate in
the case $L=K(\sqrt[p]{\pi})$ with $v_K(\pi)=1$: indeed, for any $b
\in \Z$, the element $\pi^{b/p}$ has valuation $b$ but is not a normal
basis generator. (Here $K$ must
contain a primitive $p$th root of unity for $L/K$ to be Galois.) We
also raised the question of whether the corresponding result held in
characteristic $p>0$, where the exceptional situation of maximal
ramification cannot arise.  Our question was answered by L.~Thomas
\cite{T}, who observed that general properties of group algebras of
$p$-groups in characteristic $p$ allow an elegant derivation of
integer certificates for arbitrary finite abelian $p$-groups $G$. Her
result was expressed in terms of the last break in the sequence of
ramification groups of $L/K$, but is equivalent to Theorem \ref{elder}
for totally ramified abelian $p$-extensions $G$.  Finally, Elder
\cite{E} removed the hypothesis that $G$ is abelian by expressing the
result in terms of the valuation of the different, and also gave the
converse result that no integer certificate exists if $L/K$ is not
totally ramified or is not a $p$-extension.

We end this introduction by outlining the structure of the paper. In
\S2, we review the facts we shall need from Hopf-Galois theory, and
prove several preliminary results in the case of $p$-extensions. These
show, in effect, that the relevant Hopf algebras behave similarly to
the group algebras considered in \cite{T}. In \S3 we develop some
machinery to handle extensions whose degrees are not powers of $p$. In
\cite{E}, such extensions were treated by reducing to a totally and
tamely ramified extension. For Hopf-Galois extensions, it is not clear
whether such a reduction is always possible. (Indeed, while a totally
ramified Galois extension of local fields is always soluble, the
author does not know of any reason why such an extension 
could not admit a Hopf-Galois structure in which the associated
group $N$, as in \S2 below, is insoluble.) We therefore adopt a
different approach, using a small part of the theory of modular
representations. We complete the proof of Theorem \ref{main} in
\S4. The ramification groups, which play an essential role in the
arguments of \cite{E} and \cite{T}, are not available in the
Hopf-Galois setting, but their use can be avoided by working directly
with the inverse different. Finally, in \S5, we give an example of a
family of extensions which are not Galois, but to which Theorem
\ref{main} applies.

\section{Hopf-Galois theory for $p$-extensions in characteristic $p$}

In this section, we briefly recall the description of Hopf-Galois
structures on a finite separable field extension $L/K$, and note some
properties of the Hopf algebras $H$ which arise when $[L:K]$ is a
power of $p=\chr(K)$. We do not make any use of valuations on $K$ and
$L$ in this section. 

Let $E$ be a (finite or infinite) Galois extension of $K$ containing
$L$. Set $G=\Gal(E/K)$ and $G'=\Gal(E/L)$, and let $X=G/G'$ be the set
of left cosets $gG'$ of $G'$ in $G$. Then $G$ acts by left
multiplication on $X$, giving a homomorphism $G \lra \Perm(X)$ into
the group of permutations of $X$. The main result of \cite{GP} can be
stated as follows: 
the Hopf-Galois structures on $L/K$ (up to the appropriate notion of
isomorphism) correspond bijectively to the
regular subgroups $N$ of $\Perm(X)$ which are normalised by $G$.  In
the Hopf-Galois structure corresponding to $N$, the Hopf algebra
acting on $L$ is $H=E[N]^G$, the fixed point algebra of the group
algebra $E[N]$ under the action of $G$ simultaneously on $E$ (as field
automorphisms) and on $N$ (by conjugation inside $\Perm(X)$). The Hopf
algebra operations on $H$ are the restrictions of the standard
operations on $E[N]$. We write $1_X$ for the trivial coset $G'$ in
$X$. Then there is a bijection between elements $\eta$ of $N$ and
$K$-embeddings $\sigma \from L \lra E$, given by $\eta \mapsto
\sigma_\eta$ where $\sigma_\eta(\rho)=g(\rho)$ with
$\eta^{-1}(1_X)=gG'$. The action of $H$ on $L$ can be described
explicitly as follows (see e.g.~\cite[p.~338]{NB}):
\begin{equation} \label{Hopf-action}
 \left( \sum_{\eta \in N} \lambda_\eta \eta \right)(\rho)= 
  \sum_{\eta \in N} \lambda_\eta \sigma_\eta(\rho) \mbox{ for }
  \sum_{\eta \in N} \lambda_\eta \eta  \in H \mbox{ and } \rho \in L. 
\end{equation}

\begin{remark} 
In \cite{GP}, $E$ is taken to be the the Galois closure $E_0$ of $L$
over $K$. In this case, the action of $G$ on $X$ is faithful. However,
it is clear that one may take a larger field $E$ as above: all that
changes is that $G$ need no longer act faithfully on $X$. (Indeed, the
action of $G$ on both $X$ and $L$ factors through $\Gal(E/E_0)$.) 
In the proof of Lemma \ref{non-p-power} below, it will be
convenient to take $E$ to be a finite extension of $E_0$.
\end{remark}

Let $L/K$ be $H$-Galois, where the Hopf algebra $H$ corresponds to $N$
as above. We define 
$$ t_H = \sum_{\eta \in N} \eta \in E[N]. $$
We now show that $t_H$ behaves like the trace element in a group
algebra:
 
\begin{proposition}  \label{trace}
We have  $t_H \in H$ and, for any $h \in H$, 
 $$ h t_H = t_H h = \epsilon(h) t_H, $$
where $\epsilon \from H \to K$ is the augmentation. In particular,
writing $I_H$ for the augmentation ideal $ \ker \epsilon$ of $H$, we have 
$$ I_H t_H = t_H I_H =0. $$
Also, $t_H(\rho)=\Tr_{L/K}(\rho)$ for any $\rho \in L$.
\end{proposition}
\begin{proof}
Since $N$ is normalised by $G$, each $g \in G$ permutes the elements
of $N$. Hence $t_H \in E[N]^G=H$. For any $h=\sum_{\nu \in N} \lambda_{\nu} \nu
\in H$, we have 
$$ h t_H = \sum_{\nu, \eta} \lambda_{\nu}  \nu \eta  
   = \left( \sum_\nu \lambda_\nu \right) \left( \sum_\eta \eta \right)
   = \epsilon(h) t_H. $$ 
In particular, if $h \in I_H$ then $ht_H =\epsilon(h) t_H = 0$, so
$I_Ht_H=0$. Similarly $t_H h = \epsilon(h) t_H$ and $t_H I_H =0$.
Finally, for $\rho \in L$ we have
$$ t_H(\rho)=\sum_{\eta \in N} \sigma_\eta(\rho) 
           = \Tr_{L/K}(\rho). $$
\end{proof}

\begin{remark}
Proposition \ref{trace} shows that $K \cdot t_H$  is the ideal of
(left or right) integrals of $H$. 
\end{remark}

\begin{corollary} \label{not_NBG}
If $\Tr_{L/K}(\rho)=0$ then $\rho$ cannot be a normal basis generator
for $L/K$ with respect to $H$.
\end{corollary}
\begin{proof}
If $\rho$ is a free generator for $L$ over $H$, then the annihilator
of $\rho$ in $H$ must be trivial. But if $\Tr_{L/K}(\rho)=0$ then
$\rho$ is annihilated by $t_H \neq 0$. 
\end{proof}

We next show that \cite[Proposition 7]{T} still holds in our setting:

\begin{lemma} \label{t-gen}
If $[L:K]=p^m$ for some integer $m$, then any $\rho \in L$ with
$\Tr_{L/K}(\rho) \neq 0$ is a normal basis generator for $L/K$ with
respect to $H$.
\end{lemma}
\begin{proof}
We first observe that the augmentation ideal $I_H$ is a nilpotent
ideal of $H$, since $I_H = I_{E[N]} \cap H$ and the augmentation
ideal $I_{E[N]}$ of $E[N]$ is a nilpotent ideal of $E[N]$ because
$|N|=[L:K]=p^m$. Thus $I_H$ is contained in (and in fact equals)
the Jacobson radical $J_H$ of $H$.

Now consider the $H$-submodule $M=H \cdot \rho + I_H \cdot L$ of $L$. 
Since $L$ is a free $H$-module of rank 1, and $H/I_H \cong K$, the $K$-subspace
$I_H L$ of $L$ has codimension 1. But $\rho \not \in I_H L$ since   
$\Tr_{L/K}(I_H L) = (t_H I_H)L=0$ by Proposition \ref{trace}, so
$M=L$. Since $I_H \subseteq J_H$, Nakayama's Lemma shows that $H \cdot \rho
=L$, and, comparing dimensions over $K$, we see that $\rho$ is a
free generator for the $H$-module $L$. 
\end{proof}

The next result is immediate from Corollary \ref{not_NBG} and Lemma
\ref{t-gen}

\begin{corollary}
If $[L:K]=p^m$ then $\rho \in L$ is a normal basis generator for $L/K$
with respect to $H$ if and only if $\Tr_{L/K}(\rho) \neq 0$. In
particular, the set of normal basis generators is the same for all
Hopf-Galois structures on $L/K$. 
\end{corollary}

\section{The non-$p$-power case}

As in Theorem \ref{main}, let $S/R$ be a finite extension of discrete
valuation rings, such that the corresponding extension $L/K$ of their
fields of fractions is $H$-Galois for some Hopf algebra $H$. 
We do not require $S$ and $R$ to be complete. Let
$v_L$, $v_K$ be the corresponding valuations on $L$, $K$.

\begin{lemma} \label{non-p-power}
Suppose that $[L:K]$ is not a power of $p$. Then $H$ contains nonzero
orthogonal idempotents $e_1$, $e_2$ with $e_1+e_2=1$, such that
$$ v_L( e_j \rho) \geq v_L(\rho) \mbox{ for  all } \rho \in L \mbox{
  and } j=1, 2. $$
\end{lemma}
\begin{proof}
Let $[L:K]=p^m r$ where $m \geq 0$ and where $r \geq 2$ is prime to
$p$. We have $H=E[N]^G$ where $G=\Gal(E/K)$ and, in view of the remark
before Proposition \ref{trace}, we may take $E$ to be a finite Galois
extension of $K$, containing $L$ and also containing a primitive $r$th
root of unity $\zeta_r$. Let $k'$ be the algebraic closure in $E$ of
the prime subfield $\F_p$. Thus $\zeta_r \in k'$.

Now let $t$ be the number of conjugacy classes in $N$ consisting of
elements whose order is prime to $p$. As $|N|=[L:K]$ is not a power of
$p$, we have $t \geq 2$. For any field $F$ of characteristic $p$ containing
$\zeta_r$, the group algebra $A=F[N]$ has exactly $t$
nonisomorphic simple modules \cite[\S18.2, Corollary 3]{Se-Reps}. Let
$J_A$ denote the Jacobson radical of $A$. Then the semisimple algebra
$A/J_A$ has exactly $t$ Wedderburn components, and therefore has
exactly $t$ primitive central idempotents. Since $A$ is a
finite-dimensional $F$-algebra, we may lift these idempotents  
from $A/J_A$ to $A$. Thus $A$ has exactly $t$ primitive
central idempotents, $\phi_1, \ldots, \phi_t$ say, and hence has $t$
maximal 2-sided ideals.  One of these, say the ideal $(1-\phi_1) A$
associated to $\phi_1$, is the augmentation ideal $I_A$.

Taking $F=k'$ in the previous paragraph, we obtain orthogonal
idempotents $\phi_1, \ldots, \phi_t \in k'[N]$. But $k' \subset E$,
and taking $F=E$, we find that $\phi_1, \ldots, \phi_t$ are again the
primitive central idempotents in $E[N]$. The action of $G$ on $E[N]$
permutes these idempotents, and fixes $\phi_1$ since it fixes the
augmentation ideal of $E[N]$. Hence $\phi_1 \in H$. Let $e_1=\phi_1$
and $e_2=1-\phi_1$. Then $e_1$, $e_2$ are orthogonal idempotents in $H
\cap k'[N]$ with $e_1+e_2=1$. Moreover $e_1 \neq 0$ by definition and
$e_2 \neq 0$ since $t \geq 2$.

We now show that $v_L(e_j \rho) \geq v_L( \rho)$ for $j=1$, $2$ and
for any $\rho \in L$.  Since $S/R$ is finite, $S$ is the unique
valuation ring of $L$ containing $R$. Thus each valuation ring $T$ of
$E$ containing $R$ must also contain $S$. (There may be several such
$T$ if $R$ is not complete.) Fix one of these valuation rings $T$ of
$E$, and let $v_E$ be the corresponding valuation on $E$. Then any
valuation $v'$ on $E$ with $v'(\mu)=v_E(\mu)$ for all $\mu \in K$
necessarily satisfies $v'(\rho)=v_E(\rho)$ for all $\rho \in L$.  In
particular, for each $g \in G$, the valuation $v_E \circ g$
on $E$ must have the same restriction to $L$ as $v_E$.  Thus, for each
$\eta \in N$, we have $v_E(\sigma_\eta(\rho)) = v_E(\rho)$ for all
$\rho \in L$.

For $j=1$ or $2$, let
$$ e_j = \sum_{\eta \in N} \lambda_\eta \eta  \quad \mbox{with }\lambda_\eta \in
k'. $$
Then, as $e_j \in H$, we have
$$ e_j(\rho)= \sum_{\eta \in N} \lambda_\eta \sigma_\eta(\rho) $$
by (\ref{Hopf-action}). But $\lambda_\eta$ is algebraic over $\Fp$, so
either $\lambda_\eta=0$ or $v_E(\lambda_\eta)=0$. We then have
$$ v_E(e_j \rho) \geq \min_{\eta \in N} (v_E(\lambda_\eta) +
v_E(\sigma_\eta(\rho))) \geq 0 + v_E(\rho). $$ 
As $\rho$, $e_j \rho \in L$, it follows that $v_L(e_j \rho) \geq
v_L(\rho)$ as required.
\end{proof}

We can now prove case (i) of Theorem \ref{main}(b).

\begin{corollary} \label{non-p-done}
Let $S/R$ be as in Theorem \ref{main}, and suppose that $[L:K]$ is not
a power of $p$. Then, for any $b \in \Z$, there exists some $\rho \in
L$ with $v_L(\rho)=b$ such that $\rho$ is not a normal basis generator
for $L/K$ with respect to $H$.
\end{corollary}
\begin{proof}
Take any $\rho' \in L$ with $v_L(\rho')=b$. With $e_1$, $e_2
\in H$ as in Lemma \ref{non-p-power}, we have 
$$ \rho' = e_1 \rho' + e_2 \rho', \qquad 
    v_L(e_1 \rho')\geq b, \quad  v_L(e_2 \rho') \geq b. $$
Both inequalities cannot be strict since $v_L(\rho')=b$, so without
    loss of generality we have $v_L(e_1 \rho')=b$. 
Set $\rho= e_1 \rho'$. Then $v_L(\rho)=b$ but $\rho$
cannot be a normal basis generator with respect to $H$, since $e_2
\rho=(e_2 e_1) \rho'=0$. 
\end{proof}

\section{Proof of Theorem 2}

For this section, the hypotheses of Theorem \ref{main} are in
force. In particular, $S/R$ is a finite extension of discrete
valuation rings of characteristic $p>0$, and the corresponding
extension of fields of fractions $L/K$ is separable of degree
$n$. Also, $L/K$ is $H$-Galois for some $K$-Hopf algebra $H$.

By Corollary \ref{non-p-done}, we may assume that $n=[L:K]$ is a
power of $p$. Let $e$ be the ramification index of $S/R$, let
$w=v_L(\D_{S/R})$, and let $\pi$ and $\Pi$ be uniformisers for $R$ and
$S$ respectively. By definition of the different, we have  
$$ \Tr_{L/K}(\Pi^{-w}S)  \subseteq R,  
     \qquad \Tr_{L/K}(\Pi^{-w-1} S ) \not \subseteq R, $$
and therefore
$$ \Tr_{L/K}(\Pi^{e-w}S)  \subseteq \pi R,  
     \qquad \Tr_{L/K}(\Pi^{e-w-1} S ) = R. $$
Hence there is some $x_1 \in L$ with $v_L(x_1)=e-w-1$ and
$\Tr_{L/K}(x_1)=1$. For $2 \leq i \leq e$, pick $x'_i \in L$ with
$v_L(x'_i)=e-w-i$, and set $x_i=x'_i-\Tr_{L/K}(x'_i) x_1$. Since
$\Tr_{L/K}(x'_i) \in R$ and $v_L(x'_i)<v_L(x_1)$, we have
\begin{equation} \label{x-val}
v_L( x_i ) = e-w-i \mbox{ for } 1 \leq i \leq e, 
\end{equation}
and clearly
\begin{equation} \label{x-trace}
  \Tr_{L/K}(x_i) = 
\begin{cases}
         1 & \text{if $i=1$;} \\
         0 & \text{otherwise.}
\end{cases}
\end{equation}

We first consider the totally ramified case $e=n$. Then
$x_1,\ldots, x_n$ is a $K$-basis for $L$, since the $v_L(x_i)$
represent all residue classes modulo $n$. 

Let $\rho \in L$ with
$v_L(\rho) \equiv -w-1 \pmod{n}$. We may write 
$$ \rho = \sum_{i=1}^n a_i x_i $$
with the $a_i \in K$. Then  
$v_L(\rho) = \min_i \{ nv_K(a_i)+(n-w-i)\}$. The hypothesis on $\rho$
means that the minimum must occur at $i=1$. In particular, $a_1
\neq 0$.  Then, by (\ref{x-trace}), we have 
$$ \Tr_{L/K}(\rho) = \sum_{i=1}^{n} a_i \Tr_{L/K}(x_i) = a_1 \neq 0, $$
and by Lemma \ref{t-gen}, $\rho$ is a normal basis generator for $L/K$
with respect to $H$. This completes the proof of Theorem
\ref{main}(a).

Next let $b \in \Z$ with $b \not \equiv -1-w \pmod{n}$.   
Then $b=n(s+1)-w-i$ with $2 \leq i \leq n$ and $s \in
\Z$. Set $\rho=\pi^s x_i$, so 
$v_L(\rho)=b$ by (\ref{x-val}). But $\Tr_{L/K}(\rho)=0$ by
(\ref{x-trace}), so that $\rho$ cannot be a normal basis generator by
Corollary \ref{not_NBG}. This completes the proof of Theorem
\ref{main} for totally ramified extensions. 

Finally, suppose that $S/R$ is not totally ramified. Given $b \in \Z$,
write $b=e(s+1)-w-i$ with $1 \leq i \leq e$ and $s \in \Z$. If $i \neq
1$ then $\rho=\pi^s x_i$ satisfies $v_L(\rho)=b$ and
$\Tr_{L/K}(\rho)=0$, so as before $\rho$ cannot be a normal basis
generator. It remains to consider the case $i=1$. Let $l$, $k$ be the
residue fields of $S$, $R$ respectively. Then $l/k$ has degree $f>1$
with $ef=n$. (Note, however, that $l/k$ need not be separable.) Pick
$\omega \in l$ with $\omega \not \in k$, let $\Omega \in S$ be any
element whose image in $l$ is $\omega$, and set
$$ \rho = \pi^s (\Omega-\Tr_{L/K}(x_1 \Omega)) x_1. $$
Then $\Tr_{L/K}(x_1 \Omega) \in \Tr_{L/K} (\D_{S/R}^{-1}) \subseteq
R$. Since $\omega$ and $1$ are elements of $l$ which are linearly
independent over $k$, it follows that $v_L(\Omega-\Tr_{L/K}(x_1
\Omega)) = v_L(\Omega)=0$, and hence $v_L(\rho)=es+v_L(x_1)=b$. But
once more we have $\Tr_{L/K}(\rho)=0$, so that $\rho$ cannot be a
normal basis generator for $L/K$ with respect to $H$. This concludes
the proof of Theorem \ref{main}.

\section{An example}

We end with an example of a family of extensions $L/K$ which are $H$-Galois
for a suitable Hopf algebra $H$, but which are not Galois. Theorem 2 will
give an integer certificate for normal basis generators in $L/K$,
although Theorem 1 is not applicable.

Fix a prime number $p$, and let $K=\Fp((T))$ be the field of formal
Laurent series over the finite field $\Fp$ of $p$ elements. Then $K$
is complete with respect to the discrete valuation $v_K$ such that
$v_K(T)=1$, and the valuation ring is $R=\Fp[[T]]$. Take any integer
$f \geq 2$, and set $q=p^f$. Let $b>0$ be an integer which is not divisible
by $p$, and let $\alpha \in K$ be any element with $v_K(\alpha)=-b$.
The field we consider is $L=K(\theta)$, where $\theta$ is a root of
the polynomial $g(X)=X^q-X-\alpha \in K[X]$. 

To see that $L$ is not Galois over $K$, consider the unramified
extension $F=\Fq K$ of $K$ (where $\Fq$ is the field of $q$ elements),
and let $E=LF$. Then $E$ is the splitting field of $g$ over $K$, and
the roots of $g$ in $E$ are $\{ \theta + \omega \mid \omega \in
\Fq\}$. Thus $E$ is the Galois closure of $L/K$, and it follows in
particular that $L/K$ is not Galois. We are therefore in the situation
of \S2, with $G=\Gal(E/K)$ of order $fq$, and with $G'=\Gal(E/L) \cong
\Gal(F/K)\cong \Gal(\Fq/\Fp)$ cyclic of order $f$. Moreover, $G'$ has
a normal complement $N=\Gal(E/F) \cong \Fq$ in $G$. Thus
$G \cong N \rtimes G'$ (and, since $\F_q/\F_p$ has a
normal basis, it is easy to see that any generator of $G'$ acts on $N$
with minimal polynomial $X^f-1$). In the
terminology of \cite[\S4]{GP}, $L/K$ is an almost classically Galois
extension. It therefore admits at least one Hopf-Galois structure,
namely that corresponding to the group $N$.

Now $E/F$ is totally ramified of degree $q$, and the ramification
filtration of $\Gal(E/F)$ has only one break, occurring at $b$. Hence,
by Hilbert's formula \cite[IV, Prop.~4]{Se-CL}, $v_E(\D_{E/F})=(b+1)(q-1)$. As
$E/L$ and $F/K$ are unramified, it follows that $L/K$ is totally
ramified, and, using the transitivity of the different \cite[III,
Prop.~8]{Se-CL},
that $v_L(\D_{L/K})=(b+1)(q-1)$. Thus Theorem \ref{main}(a) applies
with $w \equiv -1-b \pmod{q}$. Hence any $\rho \in L$ with $v_L(\rho)
\equiv b \pmod{q}$ is a normal basis generator with respect to {\em any}
Hopf-Galois structure on $L/K$.

Following a suggestion of the referee, we specialise this example
further. Let us take $b=q-1$ and $\alpha=T^{1-q}$. Then
$v_L(\theta)=1-q$. We obtain a uniformising parameter for $S$ by
seting $\eta=T\theta$. Then $\eta$ is a root of the Eisenstein polynomial
$X^q-T^{q-1}X-T$, so $\D_{L/K}$ is generated by $T^{q-1}$ and $w \equiv
0 \pmod{q}$. Hence any element $\rho$ of $L$ with $v_L(\rho) \equiv -1
\pmod{q}$ is a normal basis generator with respect to any Hopf-Galois
structure on $L/K$. This can easily be verified directly for
$\rho=\eta^{q-1}$ and the Hopf-Galois structure corresponding to $N$
as above. Indeed, let $\sigma_\omega$ be the element of $N=\Gal(E/F)$
corresponding to $\omega \in \F_q$, so $\sigma_\omega(\eta)=\eta +
\omega T$. We first claim that $\eta^{q-1}$ is a normal basis generator
for the Galois extension $E/F$, or equivalently, that $F[N] \cdot
\eta^{q-1}=E$. We have 
$$  \sigma_\omega(\eta^{q-1}) = (\eta+\omega T)^{q-1} = 
   \sum_{i=0}^{q-1} \eta^{q-1-i}(-\omega T)^i, $$
so the claim follows from the non-vanishing of the Vandermonde matrix
$\big((-\omega)^i\big)_{\omega \in \F_q, 0 \leq i <q}$. 
Since the $F[N]$-module $E$ is free on the generator $\eta^{q-1}$, and
$H=F[N]^G$ is a $K$-subalgebra of $F[N]$, it follows that $H \cdot
\eta^{q-1}$ has dimension $\dim_K(H)=q=[L:K]$ over $K$. But $\eta \in L$
and $H \cdot L = L$, so we must have $H \cdot \eta^{q-1}=L$. Thus
$\eta^{q-1}$ is a normal basis generator for $L/K$ over $H$, as required. 

\begin{remark}[Galois extensions] 
If we apply the preceding construction starting with $\F_q((T))$ rather
than $\F_p((T))$ (that is, we just consider the extension $E/F$ above) then
we obtain a {\em Galois} (indeed, abelian) extension of degree $q$ for
which we have given a direct verification that $\eta^{q-1}$ is a
normal basis generator. This provides an explict example of the
situation considered in \cite{T}
\end{remark}

\begin{remark}[Global examples]
We can easily adapt the above arguments to the case where 
$K$ is not complete. Let $K$ be a function field of
dimension 1 with field of constants $\Fp$, and choose any valuation
$v_K$ on $K$ which corresponds to a place of $K$ with residue field
$\F_p$. With $q$, $b$ and $\alpha$ as above, let 
$L=K(\theta)$ where $\theta^q-\theta=\alpha$. Then the extension $L/K$
has degree $q$ and is a totally ramified at $v_K$. As before, $L/K$ is
not Galois but does admit at least one Hopf-Galois structure, and
Theorem \ref{main}(a) shows that any $\rho \in
L$ with $v_L(\rho) \equiv b \pmod{q}$ is a normal basis generator for
$L/K$ with respect to any Hopf-Galois structure on $L/K$.
\end{remark}

\end{document}